\documentclass[12pt]{amsart}
\usepackage{amscd,amssymb,latexsym, mathrsfs}
\usepackage{fullpage}
\usepackage{graphicx}
\newcommand{\Mdef}[2]{\newcommand{#1}{\relax \ifmmode #2 \else $#2$\fi}}

\newcommand{\codim}{\mathrm{codim}}

\newcommand{\supp}{\mathrm{supp}}
\newcommand{\cosupp}{\mathrm{cosupp}}

\newcommand{\tensor}{\otimes}

\newcommand{\Hom}{\mathrm{Hom}}

\Mdef{\bhom}{\mathbf{\hat{H}om}}

\Mdef{\Mod}{\mathrm{mod}}

\newcommand{\st}{\; | \;}

\newtheorem{thm}{Theorem}[section]
\newtheorem{lemma}[thm]{Lemma}
\newtheorem{prop}[thm]{Proposition}
\newtheorem{cor}[thm]{Corollary}

\theoremstyle{definition}

\newtheorem{defn}[thm]{Definition}

\newtheorem{example}[thm]{Example}

\newtheorem{remark}[thm]{Remark}

\newcommand{\qqed}{\qed \\[1ex]}
\renewenvironment{proof}[1][\hspace*{-.8ex}]{\noindent {\bf Proof #1:\;}}{\qqed}

\Mdef{\PH} {\Phi^H}
\Mdef{\PK} {\Phi^K}
\Mdef{\PL} {\Phi^L}
\Mdef{\PT} {\Phi^{\T}}

\Mdef{\ef}{E{\cF}_+}
\Mdef{\etf}{\widetilde{E}{\cF}}
\Mdef{\eg}{E{G}_+}
\Mdef{\etg}{\tilde{E}{G}}
\Mdef{\infl}{\mathrm{inf}}
\Mdef{\defl}{\mathrm{def}}
\Mdef{\res}{\mathrm{res}}
\Mdef{\ind}{\mathrm{ind}}
\Mdef{\coind}{\mathrm{coind}}

\Mdef{\univ}{\mathcal{U}}

\Mdef{\Fp}{\mathbb{F}_p}
\Mdef{\Zpinfty}{\Z /p^{\infty}}
\Mdef{\Zpadic}{\Z_p^{\wedge}}

\newcommand{\bi}{\begin{itemize}}
\newcommand{\be}{\begin{enumerate}}
\newcommand{\bc}{\begin{center}}
\newcommand{\bd}{\begin{description}}
\newcommand{\ei}{\end{itemize}}
\newcommand{\ee}{\end{enumerate}}
\newcommand{\ec}{\end{center}}
\newcommand{\ed}{\end{description}}
\newcommand{\dichotomy}[2]{\left\{ \begin{array}{ll}#1\\#2 \end{array}\right.}

\newcommand{\lra}{\longrightarrow}

\newcommand{\spec}{\mathrm{Spec}}

\newcommand{\spcc}{\mathrm{Spcc}}

\Mdef{\we}{\mathbf{we}}
\Mdef{\fib}{\mathbf{fib}}
\Mdef{\cof}{\mathbf{cof}}
\Mdef{\BI}{\mathcal{BI}}

\newcommand{\fibre}{\mathrm{fibre}}

\newcommand{\colim}{\mathop{  \mathop{\mathrm {lim}} \limits_\rightarrow} \nolimits}

\Mdef{\B}{\mathbb{B}}
\Mdef{\C}{\mathbb{C}}
\Mdef{\D}{\mathbb{D}}
\Mdef{\E}{\mathbb{E}}
\Mdef{\T}{\mathbb{T}}
\Mdef{\F}{\mathbb{F}}
\Mdef{\G}{\mathbb{G}}
\Mdef{\I}{\mathbb{I}}
\Mdef{\N}{\mathbb{N}}
\Mdef{\Q}{\mathbb{Q}}
\Mdef{\R}{\mathbb{R}}
\Mdef{\bbS}{\mathbb{S}}
\Mdef{\Z}{\mathbb{Z}}

\Mdef{\bA}{\mathbb{A}}
\Mdef{\bB}{\mathbb{B}}
\Mdef{\bC}{\mathbb{C}}
\Mdef{\bD}{\mathbb{D}}
\Mdef{\bE}{\mathbb{E}}
\Mdef{\bF}{\mathbb{F}}
\Mdef{\bG}{\mathbb{G}}
\Mdef{\bH}{\mathbb{H}}
\Mdef{\bI}{\mathbb{I}}
\Mdef{\bJ}{\mathbb{J}}
\Mdef{\bK}{\mathbb{K}}
\Mdef{\bL}{\mathbb{L}}
\Mdef{\bM}{\mathbb{M}}
\Mdef{\bN}{\mathbb{N}}
\Mdef{\bO}{\mathbb{O}}
\Mdef{\bP}{\mathbb{P}}
\Mdef{\bQ}{\mathbb{Q}}
\Mdef{\bR}{\mathbb{R}}
\Mdef{\bS}{\mathbb{S}}
\Mdef{\bT}{\mathbb{T}}
\Mdef{\bU}{\mathbb{U}}
\Mdef{\bV}{\mathbb{V}}
\Mdef{\bW}{\mathbb{W}}
\Mdef{\bX}{\mathbb{X}}
\Mdef{\bY}{\mathbb{Y}}
\Mdef{\bZ}{\mathbb{Z}}

\Mdef{\cA}{\mathcal{A}}
\Mdef{\cB}{\mathcal{B}}
\Mdef{\cC}{\mathcal{C}}
\Mdef{\mcD}{\mathcal{D}} % Something funny about \cD.
\Mdef{\cE}{\mathcal{E}}
\Mdef{\cF}{\mathcal{F}}
\Mdef{\cG}{\mathcal{G}}
\Mdef{\mcH}{\mathcal{H}} % There's something funny about \cH: it 
\Mdef{\cI}{\mathcal{I}}
\Mdef{\cJ}{\mathcal{J}}
\Mdef{\cK}{\mathcal{K}}
\Mdef{\mcL}{\mathcal{L}}% There's something funny about \cL: it 
\Mdef{\cM}{\mathcal{M}}
\Mdef{\cN}{\mathcal{N}}
\Mdef{\cO}{\mathcal{O}}
\Mdef{\cP}{\mathcal{P}}
\Mdef{\cQ}{\mathcal{Q}}
\Mdef{\mcR}{\mathcal{R}}% There's something funny about \cR: it 
\Mdef{\cS}{\mathcal{S}}
\Mdef{\cT}{\mathcal{T}}
\Mdef{\cU}{\mathcal{U}}
\Mdef{\cV}{\mathcal{V}}
\Mdef{\cW}{\mathcal{W}}
\Mdef{\cX}{\mathcal{X}}
\Mdef{\cY}{\mathcal{Y}}
\Mdef{\cZ}{\mathcal{Z}}

\Mdef{\ca}{\mathcal{a}}
\Mdef{\ct}{\mathcal{t}}

\Mdef{\At}{\tilde{A}}
\Mdef{\Bt}{\tilde{B}}
\Mdef{\Ct}{\tilde{C}}
\Mdef{\Et}{\tilde{E}}
\Mdef{\Ht}{\tilde{H}}
\Mdef{\Kt}{\tilde{K}}
\Mdef{\Lt}{\tilde{L}}
\Mdef{\Mt}{\tilde{M}}
\Mdef{\Nt}{\tilde{N}}
\Mdef{\Pt}{\tilde{P}}

\Mdef{\tA}{\tilde{A}}
\Mdef{\tB}{\tilde{B}}
\Mdef{\tC}{\tilde{C}}
\Mdef{\tE}{\tilde{E}}
\Mdef{\tH}{\tilde{H}}
\Mdef{\tK}{\tilde{K}}
\Mdef{\tL}{\tilde{L}}
\Mdef{\tM}{\tilde{M}}
\Mdef{\tN}{\tilde{N}}
\Mdef{\tP}{\tilde{P}}

\Mdef{\ft}{\tilde{f}}
\Mdef{\xt}{\tilde{x}}
\Mdef{\yt}{\tilde{y}}

\Mdef{\Ab}{\overline{A}}
\Mdef{\Bb}{\overline{B}}
\Mdef{\Cb}{\overline{C}}
\Mdef{\Db}{\overline{D}}
\Mdef{\Eb}{\overline{E}}
\Mdef{\Fb}{\overline{F}}
\Mdef{\Gb}{\overline{G}}
\Mdef{\Hb}{\overline{H}}
\Mdef{\Ib}{\overline{I}}
\Mdef{\Jb}{\overline{J}}
\Mdef{\Kb}{\overline{K}}
\Mdef{\Lb}{\overline{L}}
\Mdef{\Mb}{\overline{M}}
\Mdef{\Nb}{\overline{N}}
\Mdef{\Ob}{\overline{O}}
\Mdef{\Pb}{\overline{P}}
\Mdef{\Qb}{\overline{Q}}
\Mdef{\Rb}{\overline{R}}
\Mdef{\Sb}{\overline{S}}
\Mdef{\Tb}{\overline{T}}
\Mdef{\Ub}{\overline{U}}
\Mdef{\Vb}{\overline{V}}
\Mdef{\Wb}{\overline{W}}
\Mdef{\Xb}{\overline{X}}
\Mdef{\Yb}{\overline{Y}}
\Mdef{\Zb}{\overline{Z}}

\Mdef{\db}{\overline{d}}
\Mdef{\hb}{\overline{h}}
\Mdef{\qb}{\overline{q}}
\Mdef{\rb}{\overline{r}}
\Mdef{\tb}{\overline{t}}
\Mdef{\ub}{\overline{u}}
\Mdef{\vb}{\overline{v}}

\Mdef{\hc}{\hat{c}}
\Mdef{\he}{\hat{e}}
\Mdef{\hf}{\hat{f}}
\Mdef{\hA}{\hat{A}}
\Mdef{\hH}{\hat{H}}
\Mdef{\hJ}{\hat{J}}
\Mdef{\hM}{\hat{M}}
\Mdef{\hP}{\hat{P}}
\Mdef{\hQ}{\hat{Q}}

\Mdef{\thetab}{\overline{\theta}}
\Mdef{\phib}{\overline{\phi}}

\Mdef{\uA}{\underline{A}}
\Mdef{\uB}{\underline{B}}
\Mdef{\uC}{\underline{C}}
\Mdef{\uD}{\underline{D}}

\Mdef{\bolda}{\mathbf{a}}
\Mdef{\boldb}{\mathbf{b}}
\Mdef{\bfD}{\mathbf{D}}

\Mdef{\fm}{\frak{m}}
\Mdef{\fp}{\frak{p}}
\newcommand{\fX}{\mathfrak{X}}

\Mdef{\eps}{\epsilon}

\input{xypic}
\newcommand{\cEi}{\cE^{-1}}

\renewcommand{\bd}{\mathbf{d}}
\renewcommand{\be}{\mathbf{e}}
\newcommand{\cOcF}{\cO_{\cF}}

\newcommand{\fq}{\mathfrak{q}}
\newcommand{\Rep}{\mathrm{Rep}}

\newcommand{\Rad}{R_{ad}}

\newcommand{\LcFt}{L_{\tilde{\cF}}}

\newcommand{\boldd}{\mathbf{d}}

\newcommand{\bbI}{\mathbb{I}}
\newcommand{\bbD}{\mathbb{D}}
\newcommand{\sfD}{\mathsf{D}}
\begin{document}
\title{Adelic cohomology.}

\author{J.P.C.Greenlees}
\address{Warwick Mathematics Institute 
Coventry CV4 7AL. UK.}
\email{john.greenlees@warwick.ac.uk}
\date{}

\begin{abstract}
The characteristic feature of the adeles is that they involve 
localizations of products (or equivalently restricted products of 
localizations). The point of this paper is to introduce an adelic style cohomological invariant 
of a partially ordered set with auxiliary structure which covers
several examples of established interest in commutative algebra and
stable equivariant homotopy theory. 
\end{abstract}

\thanks{I am grateful Bhargav Bhatt for suggesting the connection with
the Beilinson-Parshin complex.}
\maketitle

\tableofcontents

\newcommand{\abcat}{\mathrm{AbCat}}

\section{Introduction}
\subsection{Motivation}
The characteristic feature of the adeles is that they involve 
localizations of products (or equivalently restricted products of 
localizations). The point of this paper is to introduce an adelic style cohomological invariant 
of a partially ordered set with auxiliary structure. The 
construction covers several special cases of established interest, and 
gives a language and method of calculation in many more.  

The first example is in  commutative algebra. The Hasse
square for $\Z$ came from number theory, but there is a counterpart
for any one dimensional Noetherian ring $R$. This says that $R$ is the 
pullback of a square formed using completions at primes and
localizations at primes. Furthermore, this cube is also a
pushout. Similarly,  for a $d$-dimensional catenary Noetherian ring, the ring $R$ is
the pullback of a $(d+1)$-cube, and in fact it is also a homotopy
pullback. In our terms this states that the adelic cohomology is $R$
in degree 0.  It will be shown elsewhere \cite{adelicmodels} that this
is also the basis for  understanding the category of $R$-modules. 
A variant of this construction gives the Beilinson-Parshin adeles
\cite{Huber}. 

The second established case comes up in equivariant homotopy. If $G$ is an
$r$-dimensional torus a main result of \cite{tnqcore} states that 
the rational equivariant sphere spectrum is the pullback of an $(r+1)$-cube of
commutative ring spectra. This gives an approach to the category of
$\bbS$-modules (i.e., to the category of rational $G$-spectra). Taking homotopy,  we obtain a spectral
sequence calculating the stable homotopy groups of the sphere,  $\pi^G_*(\bbS)$. In fact if we take the fixed points of the punctured
cube, the ring spectra are all formal, and the sphere is determined by a diagram of graded
rings. This diagram of rings gives a cochain complex for the adelic
cohomology, so that the $E_2$-term of the spectral sequence for 
$\pi^G_*(\bbS)$ is the adelic cohomology. 

There are other closely related examples coming out of homotopy theory
which are not covered by the construction here. The best known of
these  is the chromatic fracture square in stable
homotopy theory, but there are many others of this type. Indeed, the
work of \cite{adelicmodels} takes the present constructions at the
level of a homotopy category and uses this as the basis of an adelic
model of a symmetric monoidal model category. The main ingredient
in this  is to show that the unit is the homotopy  pullback of a suitable
cube of rings.  In some cases the completions and localizations are
functors of the homotopy of the unit, and hence the cube gives a
spectral sequence starting with the adelic cohomology we describe
here. Curiously, the formal framework for algebra is a little more
elaborate than the homotopy theory because taking homotopy of
different types of objects is described by  different algebraic
functors. Nonetheless, the constructions here are essentially abelian
category level versions of the homotopical constructions of
\cite{adelicmodels} and provide motivation as well as calculation for
that case. 

\subsection{Context}

The basic substrate is a partially ordered set (poset) $\fX$, which
will usually be infinite.  The
order relation will always be written $\leq$.  This
needs to be endowed with additional coefficient data to define
cohomology. One of the main messages of this note is that we need one
piece of data depending  contravariantly on points $x\in \fX$ (as
completion does) and one piece of
 data depending covariantly on points $x\in \fX$  (as
 localization does); this will be illustrated by a
range of examples.

\subsection{Spectral examples}
We have in mind a number of examples arising from a tensor
triangulated category $\C$. We begin by taking the Balmer spectrum $\spcc
(\C)$, consisting of the  tensor ideals $\wp$ of the subcategory
$\C^c$ of
compact objects  (i.e., $\wp $ is
closed under completing triangles, and tensoring with an arbirary
object) which are prime (in the sense that if $x\tensor y \in \wp$
then either $x$ or $y$ is in $\wp$). To start with, $\spcc (\C)$  is a parially ordered set
under inclusion.  The formalism we need for our cohomology is not
restricted to this setting, but it will colour our
choice of terminology. 

\subsubsection{Commutative rings}
\label{subsubsec:comm}
We start with a commutative Noetherian ring $R$ and  we are interested
in the category of $R$-modules. The category $\C$ is  the derived
category $D(R)$ of $R$-modules. There is a natural bijection 
$$\mathrm{Spec}(R)\stackrel{\cong}\lra \spcc(\sfD(R))$$
where the algebraic prime $\wp_a$ corresponds to the Balmer prime
$$\wp_b=\{ M \st M_{\wp_a}\simeq 0\}. $$
This is an order reversing bijection, and we will always use the 
Balmer ordering. 

We take $\fX = \mathrm{Spec}(R)$ and note that the
the minimal elements  of $\fX$ (in the Balmer ordering) correspond to closed points.

Associated to a prime $\wp$  we have $\wp$-adic completion and
localization at  $\wp$.

\subsubsection{Rational torus-equivariant spectra}
\label{subsubsec:gq}
If $G$ is an $r$-dimensional torus we may consider the category $\C$ of
rational $G$-spectra. The category $\C^c$ is then the stable homotopy category
of rational finite $G$-complexes. Equivalently, we can take $\C$ to be
the derived category of objects from the category of
differential graded objects  of the algebraic category $\cA (G)$ of
\cite{tnq1}. 

In either case,  $\fX_a=\spcc(\C)$ is the set of closed subgroups of
$G$, but the partial order is that of 
{\em cotoral inclusion} (i.e., $K\leq H$ when $K$ is a subgroup of $H$ with
$H/K$ a torus) \cite{spcgq}. 

At the prime corresponding to the subgroup $K$, the relevant
completion is given by a function spectrum (if $K=1$ the completion of
$X$ is the cofree spectrum $F(EG_+, X)$) and the localization corresponding to  inverts the
Euler classes of complex representations $V$ with $V^K=0$. 

%We will also use the poset $\fX_c$ of connected subgroups of
%$G$. There are maps of posets $\fX_c \lra \fX_a \lra \fX_c, $
%showing $\fX_c$ is a retract of $\fX_a$. 

\subsubsection{Chromatic homotopy theory}
\label{subsubsec:chrom}
It has been shown by Hopkins and Smith \cite{HopkinsSmith} that the Balmer
spectrum of finite spectra has primes corresponding to the Morava
$K$-theories $K(n,p)$ for $0\leq n \leq \infty$ for non-zero primes
$p$, where $K(\infty, p)=H\mathbb{F}_p$ and $K(0,p)=H\Q$ independent of $p$. We write
$$\wp(p,n):=\{ X \st K(n,p)_*X=0\}, $$
and then 
$$\wp (0,p)>\wp(1, p)>\wp(2, p)>\cdots >\wp(\infty,p). $$

%\part{Algebra}
\section{Terminology}
We will be introducing some constructions that extend certain standard
ones, so it will be helpful to explain our notation and terminology
first in a familiar case. 

\subsection{Coefficient systems}
If $\fX$ is partially ordered set, a {\em coefficient system} on $\fX$
with values in an abelian category $\C$ is a functor
$M: \fX^{op} \lra \C$. A {\em dual coefficient system} on $\fX$ is a functor
$N: \fX \lra \C$. 

\begin{remark}
It is more usual to call these {\em local} coefficient systems, on the
grounds 
that a coefficient system takes the same value on all simplices (but
perhaps allows for monodromy) whereas a local system varies with the
simplex. We have simplified this for brevity.   

%(ii) In fact it may also be appropriate to let the category vary with
%the object, so that we have a generalized diagram in the sense of
%\cite{diagrams}. 
\end{remark}

\subsection{Simplicial complexes}
A {\em simplicial complex} $K$ on a vertex set $V$ is a set of non-empty finite subsets of
$V$ so that $\emptyset\neq \tau \subseteq \sigma \in K$ implies $\tau
\in K$. An element $\sigma$
of $K$ with $n+1$ vertices is said to be an {\em $n$-simplex} of $K$, and
$\fX_n$ is the set of $n$-simplices of $K$. 

\begin{remark}
Note that we have explicitly declared that the empty set is not a
simplex, to fit with our applications below. 
\end{remark}

We may think of a simplicial complex $K$ as a partially ordered set,
ordered by inclusion.

\subsection{Coefficient systems on $\fX$ and its subdivision}
Suppoes now $\fX$ is a partially ordered set, and consider the 
poset $\fX'$ of non-empty flags $F=(\wp_0\supset \cdots \supset \wp_s)$. Indeed, 
$\fX'$ is a simplicial complex whose vertex set consists of objects of $\fX$.

We note that a coefficient system $M: \fX^{op} \lra \C$
 induces a coefficient system $M_*$ on $\fX'$ by
$$M_*(\wp_0\supset \cdots \supset \wp_s)=M(\wp_0). $$
It also defines a dual coefficient system on $\fX'$ defined by 
$$M^*(\wp_0\supset \cdots \supset \wp_s)=M(\wp_s). $$
Similarly a dual coefficient system $N:\fX^{op}\lra \C$ induces
a dual coefficient system $N_*$ on $\fX'$ by using the first term in the flag,
and a coefficient system $N^*$ by using the last term. These are named
so that lower star indicates the variance is the same and an upper
star implies the variance is reversed.

\subsection{Simplicial cohomology}
\label{subsec:simpcohom}
If we have a simplicial complex $K$, the simplicial cochain complex
with coefficients in $M$ is defined by 
$$
C_{simp}^*(K ; M)=\left( 
\prod_{\sigma_0 \in K_0} M
\lra \prod_{\sigma_1 \in K_1} M 
\lra \prod_{\sigma_2 \in K_2} M \lra \cdots
\right), $$
non-zero in cohomological degrees $\geq 0$ only. The differential is
defined by $\delta=\sum_i(-1)^i\delta_i$, where 
$\delta_i$ is obtained by omitting the $i$th vertex.
% We can view 
%this as obtained from the partially ordered set $K'$ 
%by placing $M$ at each vertex (i.e., at each simplex of $K$) and 
%forming a suitable differential. The key features of $K'$ are that
%simplices have a dimension and that there is a way to assign signs
%to length 1 edges of $K'$. 

More generally, if we have a dual coefficient system  $M': K'\lra \C$
on $K'$,  we may use the same method. 
$$
H_{simp}^*(K ; M')=H^*\left( 
\prod_{\sigma_0 \in K_0} M'(\sigma_0)
\lra \prod_{\sigma_1 \in K_1} M' (\sigma_1)
\lra \prod_{\sigma_2 \in K_2} M'(\sigma_2) \lra \cdots
\right)$$
In any case, if $M$ is a coefficient system on $K$ and $N$ is a dual
coefficient system on $K$ this gives us dual coefficient systems $M^*$ and $N_*$
on $K'$ and  we may define
$$H_{simp}^*(K' ; M^*) \mbox{ and }   H_{simp}^*(K' ; N_*). $$

\begin{remark}
We will generally omit the subscript $simp$ unless required for
emphasis. 
\end{remark}

% \begin{remark}
% This same method for constructing a cochain complex applies to any poset $\fY$ and any dual coefficient
% system 
% $P: \fY \lra \C$,  provided each object has a well defined dimension
% (all chains to a minimal element are finite and of the same length)
% and provided $\fY$ `admits signs',  in the sense that  given any two
% points $a$ and $c$ differing in dimension by 2, there are only
% finitely many points $b_1, b_2, \ldots , b_n$ between them and we can assign a
% sign (or other scalar) to each edge so that whenever we have the configuration
% $$\xymatrix{
% &a&&&\\
% b_1\ar[ur] &b_2\ar[u]&b_3\ar[ul] & \cdots &b_n\ar[ulll]\\
% &c\ar[ul]\ar[u]\ar[ur]\ar[urrr]&&&
% }$$
% in three adjacent dimensions the sum of the signs along the routes
% $c\lra b_i\lra a$ is zero. 
% \end{remark}

\section{Towards adelic cohomology}

We now start with a poset $\fX$, and use the above ideas to define
cohomology  at various levels of generality. The constructions are all
familiar, but running through them is a good way to introduce the
relevant notation and  structure, and to emphasize  variance of
constructions. 

\subsection{Constant coefficients}
To start with we may form the poset $\fX'$ of flags. This is a
simplicial complex, so that given an
object $M$ in an abelian category with products we may consider the
simplicial  cochain complex of Subsection \ref{subsec:simpcohom} given by 
$$C^s(\fX' ;M)=\prod_{\wp_0>\cdots >\wp_s}M. $$
The coboundary 
$$\delta : C^s(\fX' ; M)\lra C^{s+1}(\fX' ; M)$$
is then defined as an alternating sum 
$$\delta=\sum_i(-1)^i \delta_i$$
where $\delta_i$ is induced by deleting the $i$th term in a flag. 
Taking cohomology, we have 
$$H^*(\fX'; M)=H^*(C^*(\fX'; M)). $$

\begin{remark}
Note that we have displayed the simplicial complex of flags $\fX'$
(rather than the poset $\fX$ itself)  in the notation for consistency
with ordinary usage. This avoids  ambiguity  when $\fX$ itself is already a
simplicial complex. 
\end{remark}

\subsection{Coefficient systems}
Next, we suppose that rather than a single object $M$, we have a
coefficient system 
$$M: \fX^{op} \lra \C. $$
Note that $M$ induces a dual coefficient system $M^*$ on $\fX'$. 
Accordingly, we may  then define a cochain complex on objects by 
$$C^s_{fl}(\fX ;M)=C^s(\fX';M^*)=\prod_{\wp_0>\cdots >\wp_s}M
(\wp_s),  $$
with the subscript $fl$ indicating that the complex  formed from
{\em flags} in $\fX$. The coboundary 
$$\delta : C^s_{fl}(\fX ; M)\lra C^{s+1}_{fl}(\fX ; M)$$
is again defined as an alternating sum 
$$\delta=\sum_i(-1)^i \delta_i .$$
Now, if $i<s+1$ the map $\delta_i$ is still induced by deleting the
$i$th term in a flag. However $\delta_{s+1}$ is defined to be the
product over $s$-simplices $(\wp_0>\cdots >\wp_s)$ with final vertex
$\wp_s$ of the maps
$$M(\wp_s)\lra \prod_{\wp_{s+1}<\wp_s} M(\wp_{s+1})$$
whose components are given by the functor $M$. Once again, one finds
that the composites $\delta_i\delta_j$ only depend on the vertices $i$
and $j$ omitted, and hence $\delta^2=0$. We may then define the
cohomology  with coefficient system  $M$ by the formula
$$H^*_{fl}(\fX; M)=H^*(C^*_{fl}(\fX; M)). $$
\begin{remark}
(i) If $\fX$ is itself a simplicial complex, the notation $C^*(\fX; N)$ is
only defined when $N$ is a {\em dual} coefficient system. However, if $N$ is
constant for example, we may treat it as a coefficient system and so
define both $C^*(\fX; N)$ and $C^*_{fl}(\fX; N)=C^*(\fX'; N^*)$. We
note that the subdivision map 
$$Sbd: C^*(\fX; N)\lra C^*(\fX'; N^*)=C^*_{fl}(\fX;
N^*)$$
is then a chain homotopy equivalence, so the similarity in notation
should cause no serious confusion. 

(ii) When  $\C$ is symmetric monoidal, and  $M$ is a diagram of 
commutative ring objects in $\C$,  the cohomology will be ring 
valued if the images of $\delta$ are ideals, as happens if the maps
$\delta_i$ are ring maps. 
\end{remark}

\begin{example}
\label{eg:completion}
(i) Taking $\fX=\spec (R)$ with the Balmer ordering, completion
defines a coefficient system, by 
$$M(\wp)=M_{\wp}^{\wedge}. $$

(ii) We  might take $\fX_a$ to be the poset of all closed subgroups of a
torus $G$. We then have the inflation coefficient system $R$, whose value
at $K$ is $H^*(BG/K)$.  If $L\subseteq K$ we have an inflation map 
$R(K)=H^*(BG/K)\lra H^*(BG/L)=R(L)$. In particular this applies to
cotoral inclusions $L\leq K$ in the sense of Subsubsection \ref{subsubsec:gq}.

(iii) On the other hand we may take $\fX_c$ to be the poset of connected
subgroups of a torus $G$. This gives a somewhat  a more complicated example. 

Then  we define a coefficient system $R$ as follows. At a connected
subgroup  $K$ it has value 
$$  R (K)=\cO_{\cF/K}=\prod_{\Kt \in \cF/K}H^*(BG/\Kt), $$
where $\cF/K$ is the set of subgroups $\Kt$ with identity component $K$.
To see this is a coefficient system, suppose $L\subseteq K$. 
We note that if $\Lt$ has identity component $L$ then $\Kt=\Lt\cdot K$
is the unique subgroup of $G$ so that (a) $\Kt$ has identity component
$K$ and (b) $\Kt/\Lt$ is a torus. This means that if we take the
product  over $\Kt$ with identity component $K$ of the maps
$$H^*(BG/\Kt)\lra \prod_{\Lt <\Kt} H^*(BG/\Lt)$$
then we get precisely a map 
$$R(K)\lra R(L) $$
as required. 

We note that the map $q: \fX_a \lra \fX_c$ taking a subgroup to its
identity component has the requisite
properties that the coefficient system $R$ on $\fX_a$ gives the 
coefficient system $\cO_{\cF}=q_!R$ in the notation of \cite{AGs}. 

(iv) 
If we start with a spectrum $M$,  then
for a prime $\wp (p, n)$ we obtain a coefficient system by taking
the homotopy of the Bousfield localization at the
corresponding Morava $K$-theory: 
$$M(\wp (p,n))=\pi_*(L_{K (p,n)}M). $$

\end{example}

\subsection{$\fX$-collections of localizations}
Rather than just consider the individual  localizations or
completions, we will consider collections indexed by the poset $\fX$.

\begin{defn}
(i) An {\em idempotent localization} $L: \C \lra \C$ is an idempotent
monad. It consists of the functor $L$,
together with a natural transformation $\eta : 1\lra L$  giving a natural 
equivalence 
$$L\eta =\eta_L: L=L\circ 1 \stackrel{\cong}\lra 
L\circ L$$

%(ii) An {\em  $\fX$-compatible systems of localizations} is a collection
%of functors $A_{\wp}:\C \lra \C$ with natural transformations $\eta : 1\lra A_{\wp}$ so that 
%whenever $\wp_1\geq \wp_2$ the diagram 
%$$\xymatrix{
%1 \rto^{\eta}\dto_{\eta}  & A_{\wp_2}\dto^{\eta}\\
%A_{\wp_1}\rto_{A_{\wp_1}(\eta)}  & A_{\wp_1}A_{\wp_2}
%}$$
%commutes. 

(ii) An  $\fX$-collection of localizations is a collection $\{
L_{\wp}\}$ of localizations for $\wp\in \fX$. An $\fX$-collection  is said to be {\em 
 left absorbative}  if, whenever $\wp_1\geq \wp_2$, the  natural map 
$$A_{\wp_1}(\eta_{\wp_2}): A_{\wp_1}=A_{\wp_1}\circ 
1 \stackrel{\cong}\lra 
A_{\wp_1}\circ 
A_{\wp_2}$$
is an isomorphism. An $\fX$-collection  is {\em 
 right absorbative}  if, whenever $\wp_1\geq \wp_2$, the  natural map 
$$\eta_{\wp_1}: A_{\wp_2} \stackrel{\cong}\lra 
A_{\wp_1}\circ 
A_{\wp_2}$$
is an isomorphism. 
\end{defn}

\begin{lemma}
(i)  A left absorbative $\fX$-collection $L_{\bullet}$ gives idempotent localizations 
$L_{\wp}$ which fit together to give  a functor 
$$\fX \lra [\C, \C]$$ 
(i.e.,  an inequality $\wp_2\leq \wp_1$ gives rise to a natural 
transformations  $L_{\wp_2}\lra L_{\wp_1}$, and these are closed under 
composition). 

(ii)  A right absorbative $\fX$-collection $\Lambda_{\bullet}$ gives idempotent localizations 
$\Lambda_{\wp}$ which fit together to give  a functor 
$$\fX^{op} \lra [\C, \C]$$ 
(i.e.,  an inequality $\wp_2\leq \wp_1$ gives rise to a natural 
transformations  $\Lambda_{\wp_1}\lra \Lambda_{\wp_2}$, and these are closed under 
composition). 
\end{lemma}

\begin{remark}
(i) Following Part (i) we will usually refer to a left absorbative 
 system  as a {\em dual  system of
  localizations}. 

The motivating example is the collection of localizations at a prime 
in commutative algebra. Writing $\wp_a$ for prime ideal in the 
algebraic sense, and $\wp_b$ for the associated Balmer prime, we have 
 $L_{\wp_b}M=M_{\wp_a}$. The variance is covariant in $\wp_b$.

We will usually use the letter $L$ for a dual system of 
localizations.

(ii) Following Part (ii) we refer to a right absorbative system  
as a {\em (direct) system of localizations}.
 
The motivating example is the collection of completions at a prime 
in commutative algebra. Writing $\wp_a$ for prime ideal in the 
algebraic sense, and $\wp_b$ for the associated Balmer prime, we have 
 $\Lambda_{\wp_b}M=M_{\wp_a}^{\wedge}$. The variance is contravariant in $\wp_b$.

We will usually use the letter $\Lambda$ for a direct system of 
localizations. 

(iii) The composite of idempotent localizations (such as
$A_{\wp}=L_{\wp}\Lambda_{\wp}$ or $\Lambda_{\wp}L_{\wp}$) need not be an
idempotent localization, but it does come equipped with a natural
transformation  $1\lra A_{\wp}$, which will be the essential input into
constructing the adelic cochain complex below. 
%(iii) The systems $L_{\wp}$ and $\Lambda_{\wp}$ are both absorbative (on
%different sides), but not all relevant
%$\fX$-compatible systems  have this property. Indeed, 
%a composite of $\fX$-compatible systems is an $\fX$-compatible
%system. For example given  a dual coefficient system of localizations $L_{\wp}$  and
%a coefficient system of localizations $\Lambda_{\wp}$, we may consider
%their composites. 

%Both the composites $L_{\wp}\Lambda_{\wp}$ and $\Lambda_{\wp}L_{\wp}$ give
%$\fX$-compatible systems, but they are typically neither left nor right
%absorbative. 

(iv) Given an object $M$ of $\C$, any coefficient system $\Lambda_{\bullet}$ of localizations defines a  coefficient system 
$$\Lambda_{\bullet}M: \fX^{op} \lra \C. $$
Similarly a  dual system $L_{\bullet}$ of localizations, defines a 
dual coefficient system 
$$L_{\bullet}M: \fX \lra \C. $$
\end{remark}

\subsection{The dual system from commutative algebra}
 The simplest type of example might come from taking
$\fX=\spec (R)$ with the Balmer partial order (i.e., the reverse of
the classical order), and defining the system of localizations  
$$L_{\wp}M:=M_{\wp}$$
to be classical localization. Localization $M_{\wp}$ simply
inverts the elements of the complement $\wp^c$ so that  if $\wp_1\geq
\wp_2$ in the Balmer order, 
we have $(M_{\wp_2})_{\wp_1}=M_{\wp_1}$.

\subsection{Adelic cochains and adelic cohomology}
To define adelic cohomology we need a coefficent system $M$ and an 
$\fX$-collection  $A$. In this subsection we define the adelic cochain complex $C^*_{ad}(\fX ;A, M)$ and the adelic cohomology 
$$H^*_{ad}(\fX ;A, M) :=H^*(C^*_{ad}(\fX ;A, M)). $$

\begin{defn}
\label{defn:adeliccochains}
For $s\geq 0$, the {\em adelic $s$-cochains} may be thought of as functions on $s$-simplices of the poset $\fX$:
$$C^{s}(\fX;A, M)= 
\prod_{\wp_0}A_{\wp_0}
\prod_{\wp_1<\wp_0,}A_{\wp_{1}} \prod_{\wp_2<\wp_1}\cdots 
\prod_{\wp_{s-1}<\wp_{s-2}}A_{\wp_{s-1}}\prod_{\wp_s<\wp_{s-1}}A_{\wp_s}M (\wp_s).$$
\end{defn}

The differential
$$\delta : C^{s}(\fX; A,M)\lra C^{s+1}(\fX; A,M).$$
is given as a sum $\delta =\sum_i(-1)^i \delta_i$, where
$\delta_i$ is based on omitting the $i$th term in an $(s+1)$-flag.  
In more detail, if  $i<s+1$, we define
$$M_{i+1}(\wp_i)=
\prod_{\wp_{i+1}<\wp_i }A_{\wp_{i+1}} \prod_{\wp_{i+2}<\wp_{i+1}}\cdots 
A_{\wp_{s-1}}\prod_{\wp_{s}<\wp_{s-1}}A_{\wp_{s}}\prod_{\wp_{s+1}<\wp_{s}}A_{\wp_{s+1}}M (\wp_{s+1}). $$
(this  is just  $C^{s-i}(\Lambda (\wp_i)^*; A,M)$ in the
previous notation, where $\Lambda (\wp_i)$ consists of specializations
of $\wp_i$ and the star indicates that $\wp_i$ itself is omitted). 
Then the map $\delta_i$ is simply given by taking 
$$M_{i+1}(\wp_i)\lra A_{\wp_i}M_{i+1}(\wp_i)$$
 at the $i$th spot, and then applying the same sequence of products and
localizations to both domain and codomain. 

If $i=s+1$ we take the map
$$M(\wp_s) \lra \prod_{\wp_{s+1}< \wp_s}M(\wp_{s+1})\lra \prod_{\wp_{s+1}< \wp_s}A_{\wp_{s+1}}M(\wp_{s+1})$$
with components $M(\wp_s)\lra M(\wp_{s+1})\lra A_{\wp_{s+1}}M(\wp_{s+1})$ given by the coefficient
system, and then apply $A_{\wp_s}$ and the same sequence of products and
localizations to both domain and codomain. 

To see we get a cochain complex we need only observe that the
composite of two $\delta_i$s depends only on the dimensions omitted. 
More precisely, for 
$$C^{s-2}\stackrel{\delta}\lra C^{s-1}\stackrel{\delta}\lra C^{s}$$ 
if the numbers omitted are $0\leq a<b\leq s$, then we may omit $a$ and $b$
in either order and we need to know that $\delta_a\delta_b=\delta_{b-1}\delta_a$. 

If $a<b<s$ then the verification is immediate from the naturality of
the transformations $\eta_{\wp}$, together with the categorical properties
of the product. 

If $b=s$ there are two cases. The simplest is when  $a<s-2$. Then the diagram 
$$\xymatrix{
A_{\wp_{a+1}}\cdots A_{\wp_{s-1}}M(\wp_{s-1}) \rto \dto&
A_{\wp_{a+1}}\cdots A_{\wp_{s-1}}A_{\wp_s}M(\wp_{s})  \dto \\
A_{\wp_a} A_{\wp_{a+1}}\cdots A_{\wp_{s-1}}M(\wp_{s-1}) \rto &
A_{\wp_a}A_{\wp_{a+1}}\cdots A_{\wp_{s-1}}A_{\wp_s}M(\wp_{s})
}$$
commutes since $\eta: 1\lra A_{\wp_a}$ is a natural
transformation. The required commutation then  follows from the categorical 
properties of the product. 
  
The case $b=s-1, a=s$ is the most complicated. We will abbreviate
$M(\wp_s)=M(s)$ and $A_{\wp_s}=A_s$ for readability. The following
diagram has $A_0A_1\ldots A_{n-2}$ applied to it. 
$$
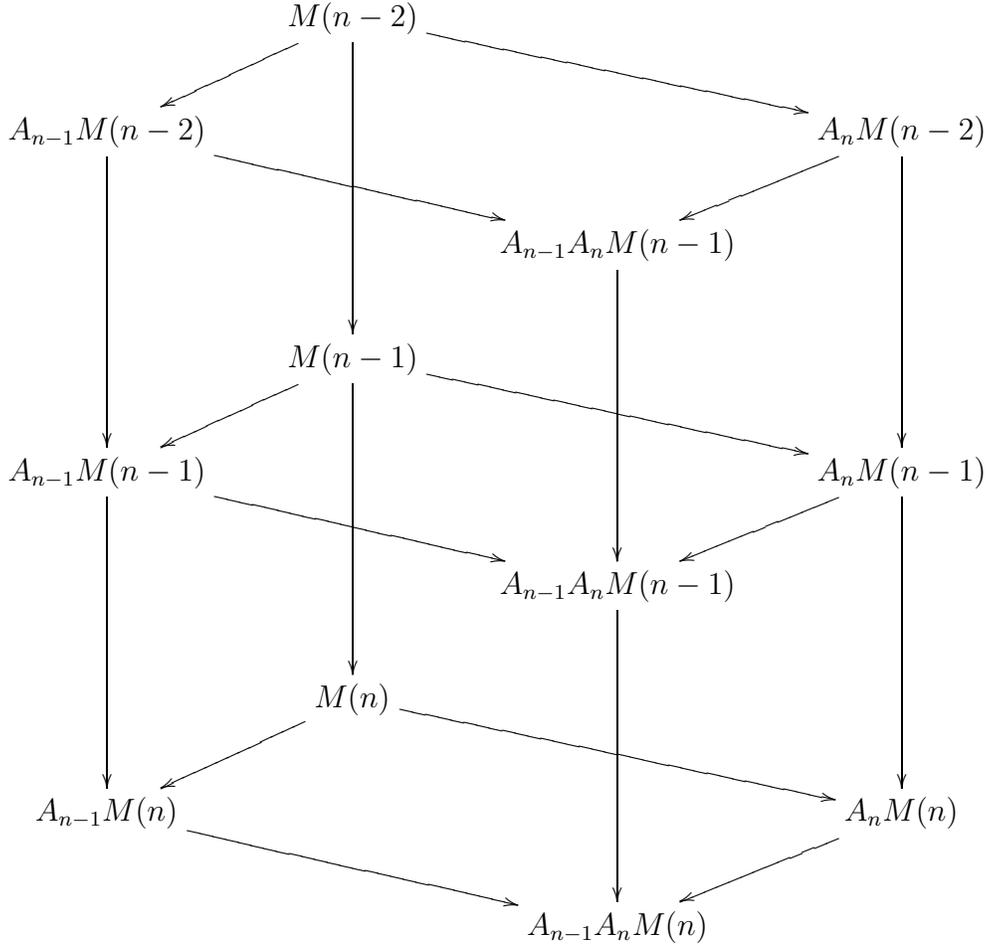
\xymatrix{
&M(n-2)\ar[dl] \ar[drr] \ar[ddd]&&\\
A_{n-1}M(n-2)\ar[drr]\ar[ddd]&&&A_nM(n-2)\ar[dl] \ar[ddd]\\
&&A_{n-1}A_nM(n-1)\ar[ddd]&\\
&M(n-1)\ar[dl] \ar[drr] \ar[ddd]&&\\
A_{n-1}M(n-1)\ar[drr]\ar[ddd]&&&A_nM(n-1)\ar[dl] \ar[ddd]\\
&&A_{n-1}A_nM(n-1)\ar[ddd]&\\
&M(n)\ar[dl] \ar[drr] &&\\
A_{n-1}M(n)\ar[drr]&&&A_nM(n)\ar[dl] \\
&&A_{n-1}A_nM(n)&}$$
The left and right faces commute since the unit for $A_{n-1}$ is a natural 
transformation.
The front and back faces commute since the unit for $A_{n}$ is a natural 
transformation.
The top and bottom faces commute because the unit for $A_{n-1}$ is a natural
transformation. 
The relevant square involves $M(n-2), A_{n-1}M(n-1), A_nM(n)$ and
$A_{n-1}A_nM(n)$.  The required commutation then  follows from the categorical 
properties of the product.

\subsection{Local meromorphic coefficients}
In this subsection we focus on the important example
where the $\fX$-collection $L$ is left absorbative (i.e., it is 
a dual system of localizations).

If the poset is finite, the products in the definition of the cochain
complex are finite. This means that all the localizations can be
collected  on the value 
$M(\wp_s)$  of the coefficient system:
$$L_{\wp_0}\cdots L_{\wp_s}M(\wp_s)=L_{\wp_0}M(\wp_s). $$
Actually, we need to deal  with infinite posets, so the localizations will usually not commute
with the products. Nonetheless, we think of $L_{\wp_0}$ as specifying
`permitted denominators in the completed stalk $M(\wp_s)$', and arrange
terminology accordingly. 

\begin{defn}
 An  {\em adelic coefficient system} is a local
coefficient system $M$, together with a dual system of localizations $L$.
%which are compatible: $M(\wp )$ is $\wp$-local in the sense that the natural map 
%$$M(\wp)\stackrel{\cong}\lra L_{\wp}M(\wp)$$
%is an isomorphism. 
\end{defn}

There are two degenerate cases. If $M$ is constant, we just have a
dual coefficient system $L_{\wp}M$. If the dual system of localizations 
all consist of the identity functor, we just have a  coefficient
system $M(\wp)$. 

\begin{example}
The motivating example of an adelic coefficient system  is  
$\fX=\spec (R)$ with the Balmer partial order (i.e., the reverse of
the classical order). We then take the completion coefficient
system $M (\wp)=M_{\wp}^{\wedge}$, and the dual system of localizations 
$$L_{\wp} N:=N_{\wp}$$
defined to be classical localization. 
 \end{example}

\section{Catenary posets $\protect \fX$}
If the poset $\fX$ is well behaved we can organize the adelic complex
into a cube. This can be helpful in examples, but it is an unnecessary detour in
the general development. 
\subsection{Dimensions}
We want to suppose that the poset $\fX$ is catenary in the sense that
for  any prime $\wp$
there is a bound on the length $s$ of flags
$$F=\left( \wp=\wp_0\supset \wp_1\supset \ldots \supset \wp_s \right)$$
of primes, and all maximal chains of this form starting at
$\wp$ have the same length. If the displayed chain is a maximal chain
starting at $\wp=\wp_0$ we will say that $\wp$ is of {\em dimension }
$s$. Thus closed points (Balmer-minimal primes) are of dimension $0$.

%The {\em component primes} are the Balmer-maximal primes. The
%{\em components} are the sets $\Lambda (\wp)$ for $\wp$ Balmer-maximal.  

%If there is a bound on the dimension of all primes, we say 
%$$\dim (\fX)=\max \{ \dim (\wp)\st \wp \in \fX\},  $$
%and otherwise we say $\dim (\fX)=\infty$.
The dimension of $\fX$ is defined by 
$$r:=\dim (\fX):=\max\{ \dim (\wp)\st \wp \in \fX\}, $$
which may be finite or $\infty$.

We write $\Delta^r:=\{ 0, 1, \ldots, r\}$ if $r$ is finite and
$\Delta^{\infty}=\mathbb{N}$. 
The dimension function $\dim: \fX\lra \Delta^r $ is a
function of posets, and therefore induces a function on flags
$$\dim :\fX'\lra (\Delta^r)'. $$
When $r$ is finite, we will usually identify $ (\Delta^r)'$ (the set
of non-empty subsets of $\{ 0,
1, \ldots , r\}$)  with the  punctured $r$-cube by identifying a subset with its
characteristic function. 
More concretely any chain of the form displayed can be viewed as an
$s$-simplex of $\fX$, and we write
$$\dim (F)=(d_0>d_1>\ldots >d_s); $$
the flag is maximal if $d_i=s-i$ where $\dim (\wp)=s$.

\begin{example}
If $R$ is a catenary commutative Noetherian ring and $\wp_a$ is a  prime in the
algebraic sense, with associated Balmer prime $\wp_b$ then 
$\dim (\wp_b)=\dim (A/\wp_a).$  The dimension of $\fX=\spec (R)$ is
the usual Krull dimension of $R$. 

%In this case there are finitely many minimal classical primes (i.e.,
%Balmer maximal primes) and hence finitely many components. 
%On the other hand there are usually infinitely many 
%primes of other dimensions. 
\end{example}

\begin{example}
If $G$ is a torus then 
$\dim (\wp_K)=\dim (K).$ 
In this case there is a unique prime $\wp_G$ which is Balmer-maximal (i.e.,
corresponding to an irreducible component), and infinitely
many closed points $\wp_F$ where $F$ is finite. The dimension of the
poset of subgroups is the dimension of $G$.
\end{example}

\subsection{Collecting cochains}
If $\fX$ is catenary, we may then divide $s$-simplices into those of different dimensions
 $d_0>d_1>\cdots >d_s$ of the vertices of the  $s$-simplex. 
Provided $\fX$ is finite dimensional, 
there are only finitely many
possible dimension vectors of  $s$-simplices. Since $A_{\wp}$ is
additive and the product is a categorical sum
$$A_{\wp}(M\times N)=A_{\wp}M \times A_{\wp}N,  $$
 we may then break cochains up by dimension. 

Thus 
$$C^s(\fX; A,M )=\prod_{d_0>d_1> \cdots >d_s} C^{d_0>d_1>\cdots
>d_s}(\fX ; A, M)$$
and
\begin{multline*}
C^{d_0>d_1>\cdots
>d_s}(\fX;A, M)= \\
\prod_{\dim \wp_0=d_0}A_{\wp_0}
\prod_{\wp_1<\wp_0, \dim  \wp_1=d_1}A_{\wp_{1}}\cdots 
\prod_{\wp_{s-1}<\wp_{s-2}, \dim
\wp_s=d_{s-1}}A_{\wp_{s-1}}\prod_{\wp_s<\wp_{s-1}, \dim
\wp_s=d_s}A_{\wp_s}M (\wp_s). 
\end{multline*}
Since the dimension of a face of a flag is the corresponding face of
its dimension, the functions $\delta_i$ are compatible with this
decomposition. 

Writing $\bd=(d_0> \cdots >d_s)$, thought of as a face of $\Delta^r$,  
 we may think of 
$$\bd \mapsto C^{\bd}(\fX; A,M)$$
as a dual coefficient system on the subdivision of the simplex
$\Delta^r$. We will often display it as a diagram on the punctured $(r+1)$-cube.
The   cochain complex $C^*(\fX; A,M)$ is  obtained
from the diagram by totalizing. 

\section{The classical Hasse square}
We suppose given a catenary Noetherian commutative ring $R$.
There are two interesting points here. The first is the straighforward
utility: we will show (Theorem \ref{thm:classicaladelic}) that the adelic cohomology is entirely in degree 0 where it is equal to the ring
 (this corresponds to the fact that the classical Hasse square is both
 a  pullback and a pushout). 

Deferring discussion of coefficients,  we have
$$H^*_{ad}(\spec (R); L, M)=H^0_{ad}(\spec (R); L,M)=R. $$
The second striking thing is that there are numerous different choices
of coefficients $L,M$ for which this is true. We could crudely say
that they are all based on localization and completion at primes
$\wp$. Our preferred variation is that the input
is the rings $L_{\wp}\Lambda_{\wp}R=(R_{\wp}^{\wedge})_{\wp}$ and the
Beilinson-Parshin approach \cite{Huber} is based on 
the rings $\Lambda_{\wp}L_{\wp}R=(R_{\wp})_{\wp}^{\wedge}$. Even
within these two categories there is some variation in the chain
complex. 

\subsection{Completion}
The dual system of localizations is $L_{\wp}M=M_{\wp}$ but there are several 
variations on what is done with it. In this subsection we consider the
choice of whether $\Lambda_{\wp}R=R_{\wp}^{\wedge}$ (which we view as
the true adelic approach) or $\Lambda_{\wp}=id$ for non-closed points.

This binary choice can each be further multiplied. 
We may divide the poset 
$(\Delta^r)'$ into an initial part $I$ and the complementary terminal 
part $J$, where there are no maps from a point of $J$ to a point of $I$. 
Then  $\Lambda'_{\wp}=id$ at points of $J$ and 
$\Lambda'_{\wp}M=M_{\wp}^{\wedge}$ for points of $I$. 
The theory that has been called `isotropic' in \cite{tnqcore} is 
the one in which $I$ consists of the finite subgroups, and we call it `local' in the present context. We
restrict the use of $\Lambda'$ to this case. 

\subsection{Products}
The second choice is about the size of the products. 
At dimension $d_0> \cdots >d_s$,  if $d_s=0$ and 
for a fixed prime $\wp_{s-1}$ of dimension $d_{s-1}$ the choice is 
whether we   use  
$$\prod_{\fm\leq \wp_{s-1}}R_{\fm}^{\wedge} \mbox{ or }
\prod_{\mathrm{all}\; \; \fm}R_{\fm}^{\wedge}. $$
The first follows Definition \ref{defn:adeliccochains}, which is
formally most natural and most geometric. The second is what 
is used in equivariant topology \cite{tnqcore}, where it was forced by
the need to have a commutative ring spectrum at that point.  

Again this binary choice may be further multiplied. We could  replace the product over
primes contained in $\wp_{i-1}$ by a larger product even when $d_s\neq 0$, 
provided this is compatible with the maps. We will not attempt to axiomatize this in
the absence of applications.  

\subsection{Beilinson-Parshin adeles}
The description in \cite[Proposition 2.1.1]{Huber} is inductive, but it amounts to taking 
the  coefficient system $M(\wp)=R$ and the system of 
localizations $\Lambda_{\wp}L_{\wp}$: 
$$H^*_{BP}(R)=H^*_{ad}(\spec(R); \Lambda L, R). $$

\subsection{The statement}
In all of the situations identified above, the adelic cohomology
recovers the ring in degree 0. 

\begin{thm}
\label{thm:classicaladelic}
If $R$ is a commutative Noetherian catenary ring and we take one  the
three coefficient systems $(L, \Lambda R)$, 
$(L, \Lambda' R)$, $(\Lambda L, R)$ 
 then there is no
higher cohomology and $H^0$ is the original ring: 
$$H_{ad}^*(\spec(R)) =R. $$
This also holds when all products involving maximal ideals involve all
maximal ideals. 
\end{thm}

\begin{proof}
For Beilinson-Parshin adeles, this is \cite[Theorem 4.1.1]{Huber}, but
our proof in the other cases gives an alternative proof. 

Our method is to establish that in the derived category, the ring 
 $R$ is the homotopy pullback of a suitable cube of rings following
 the method of  \cite{tnqcore} (see Section \ref{sec:adeliccube}):
this is done in Section \ref{sec:adeliccube}. We then apply  the results of
\cite{GMcomp} to see that the 0th left derived functor of completion for 
the Noetherian ring $R$ is the ordinary completion and the higher
derived functors vanish; it then follows that the derived completion
coincides with the ordinary completion. 
\end{proof}

\section{Subdivision and constant coefficients}
We observe here that the adelic cohomology of  a constant coefficient
system $M$, can sometimes be calculated from a much smaller cochain
complex. Indeed, the case that we will use is when $\fX$ is itself a
simplicial complex, and the result is simply that the cohomology of 
a complex and its subdivision agree. 

\subsection{Simplicial cohomology}
By the time we have explained the statement, the reader will agree
with it, but there are several notational delicacies worth clarifying
on the way. 

We are given an object $M$ in $\C$ (defining a constant coefficient
system) and a dual system   $L$ of localizations. This gives 
a dual coefficient system  $LM : \fX\lra \C$, and a corresponding dual
coefficient system $LM_*: \fX'\lra \C$ on the flag complex. If $\fX$ is itself a simplicial
complex, the simplicial cohomology of $\fX$
is defined by 
$$
H_{simp}^*(\fX ; LM)=H^*\left( 
\prod_{\sigma_0 \in \fX_0} L_{\sigma_0}M
\lra \prod_{\sigma_1 \in \fX_1} L_{\sigma_1}M
\lra \prod_{\sigma_2 \in \fX_2} L_{\sigma_2}M\lra \cdots
\right)$$

\begin{lemma}
\label{lem:adelicconstcoeffs}
If the partially ordered set $\fX$ is a finite simplicial complex then 
$$H^*_{ad}(\fX; L, M) \cong H_{simp}^*(\fX; LM)$$ 
\end{lemma}

\begin{proof}
Using the fact that localizations commute with finite products, the adelic $s$-cochains are 
$$\prod_{\sigma_0}L_{\sigma_0}
\prod_{\sigma_1\subset \sigma_0 }L_{\sigma_1}
\prod_{\sigma_2\subset \sigma_1 }L_{\sigma_2}
\cdots 
\prod_{\sigma_s\subset \sigma_{s-1} }L_{\sigma_s}M
=\prod_{\sigma_0\supset \sigma_1\supset \cdots \supset \sigma_s
}L_{\sigma_0}M. $$
The cohomology of this is the cohomology of the {\em subdivided} complex $\fX'$:
$$H^*_{ad}(\fX; L, M)=H^*_{simp}(\fX'; LM_*)$$
Finally, simplicial homology is invariant under passage to
subdivisions. 
$$H_{simp}^*(\fX; LM)\cong H_{simp}^*(\fX'; LM_*). $$
\end{proof}

\subsection{The Cech complex}
One instance of Lemma \ref{lem:adelicconstcoeffs}
 is very important to us. 

 We suppose given a ring $R$ consider
some elements $x_a$ as the subscript $a$ runs through a   set $A$. Now
we take $\fX=\Delta(A)$
to be the partially ordered set of non-empty finite subsets of $A$,
and write $\fX^+$ for the poset of all finite subsets. If $A$ is a
finite set with $r$ elements $\fX^+$ is the $r$-cube and $\fX=\Delta
(A)$ is the punctured $r$-cube. 

An $R$-module  $M$ defines a constant coefficient system on $\fX$. 
It is then natural to consider the dual coefficient system $M_{A}$ on $\fX^+$, whose value on a finite  subset 
$\sigma$  of $A$ is  the localization $M[1/x_{\sigma}]$, 
where $x_{\sigma}=\prod_{a \in \sigma}x_a$.  If $A$ is
finite with $r$ elements,  $\fX^+$ is an $r$-cube, and the diagram $M$
is commutative so we may totalize
this diagram to get the usual stable Koszul  complex.

A simple example will make the result clear, and explain our care with
the empty set. 

\begin{example}
If 
$A$ has just  $2$ elements, $\fX=\{ \{x\}, \{y\}, \{x,y\}\}$ has two vertices and one edge, and
the dual coefficient system $M_A^+$ on $\fX^+$ is 
$$\diagram 
M \rto \dto & M[\frac{1}{x}]\dto \\
M [\frac{1}{y}]\rto & M[\frac{1}{xy}]
\enddiagram$$
which is totalized to 
$$M\lra M [\frac{1}{x}]\oplus M[\frac{1}{y}]\lra M[\frac{1}{xy}], $$
whose simplicial cohomology 
is (by definition) the local cohomology 
$$H^*_{simp}(\Delta (A)^+; M_A^+)=H^*_{(x,y)}(R;M).$$
 If we
restrict attention to $\fX$, then we omit the top left entry and
regrade so that we obtain the Cech cohomology 
$$H^*_{simp}(\Delta (A); M_A)=\check{C}^*_{(x,y)}(R; M).$$ 

For the adelic cochain complex, we note $\fX'$ has three vertices
(the two $\{x\}, \{y\}$  of dimension 0 correspond to the vertices of
$\fX$,  and the one $\{x,y\}$ of dimension 1 corresponds to the whole of $\fX$), and
two edges ($\{x,y\}\supset \{x\}$ and $\{x,y\}\supset \{y\}$). The dual
coefficient system $M_A^+$ on $\fX^+$ induces $(M_A^+)_*$ on
$(\fX')^+$.  Using the
decomposition by  dimension vectors, we obtain the augmented complex
$$\diagram 
M \rto \dto & M[\frac{1}{xy}]\dto \\
M [\frac{1}{x}]\oplus M[\frac{1}{y}]\rto &
M [\frac{1}{xy}]\oplus M[\frac{1}{xy}]
\enddiagram$$
Lemma \ref{lem:adelicconstcoeffs} shows the cohomology of the
complexes obtained from the two displayed squares are the same, 
and in fact it is easy to construct a homotopy equivalence. 
\end{example}

For a general set $A$ of variables in a Noetherian ring $R$, Lemma \ref{lem:adelicconstcoeffs} 
 gives the following calculation. 
\begin{cor}
With $M$ an $R$-module and $L_A$ being the localization away from a set
$A$ of elements of $R$, we have 
$$H^*_{ad}(\Delta (A)^+; L_A, M)\cong H_{(x_a\st a\in A)}^*(R; M)$$
$$H^*_{ad}(\Delta (A); L_A, M)\cong \check{C}_{(x_a\st a\in A)}^*(R; M)$$
\end{cor}

Note that this also shows the adelic cohomology only depends on the radical of the ideal
$I(A)=(x_a\st a\in A)$. 

\begin{proof}
We will discuss the local cohomology case for definiteness, but a precisely
similar argument applies to \v{C}ech cohomology. 

The case when $A$ is finite  is given by Lemma
\ref{lem:adelicconstcoeffs} together applied to the adelic coefficient
system $(L_A,M)$.  

If we add one element to $A$ to form $B=A\cup \{b\}$, we obtain a
commutative square 
$$\xymatrix{
H^*_{ad}(\Delta (A)^+; L_A, M) \rto^{\cong}\dto&H_{(x_a\st a\in A)}^*(R;
M)\dto \\
H^*_{ad}(\Delta (A)^+; L_A, M) \rto^{\cong}&H_{(x_a\st a\in A)}^*(R;
M)
}$$
Because the ring is Noetherian, any such chain of adding elements
eventually gives vertical isomorphisms.  The choice of chain is not
important since any two particular collections
of variables $A_1, A_2$ may be compared to $A_1\coprod A_2$. 
\end{proof}

\section{Varying the ambient category}
\label{sec:varcat}
The definition of adelic cohomology above fails to cover some
important examples, so we introduce a more flexible context. 

\subsection{Arrivals from homotopy theory}
One motivation for this work is that homotopy theory provides a rich
source of examples. As described in \cite{adelicmodels} we may obtain
a coefficient system $M(\wp)$ and a dual system of localizations
$A_{\wp}$ in a category $\C$ of homotopical origin. This means that we can apply homotopy (or some other homology theory) to
obtain a coefficient system $\pi_*(M(\wp))$ in some abelian category
$\cA$. However, there is no reason to expect  $\pi_*(L_{\fq}M(\wp))$ to be a
functor of $\pi_*(M(\wp))$, so we do not automatically get a
dual  system of localizations.

Nonetheless, under catenary and finite dimensionality assumptions we
obtain an  $(r+1)$-cube in $\C$, and applying homotopy gives a cube in
$\cA$, from which we may obtain a chain complex $C^*_{ad}(\fX; \pi_*, 
L, M)$ with 
$$C^{\boldd}_{ad}(\fX; \pi_*, 
L, M)=\pi_*(C^{\boldd}_{ad}(\fX; L, M))=\pi_*(
\prod_{\wp_0}L_{\wp_0}\prod_{\wp_1}L_{\wp_1}\cdots
\prod_{\wp_s}L_{\wp_s}M(\wp_s)). $$

We would like to consider cases in which this chain complex is an
example of the adelic complex of Definition \ref{defn:adeliccochains}. Of
course one example is that from commutative algebra, but many
examples do not fit this pattern. On the other hand, a small variation
will cover some more examples, and it is the purpose of this section
to introduce the variation.

\subsection{Relative localization}
It may happen that $\pi_*(M(\wp))$ takes values in an abelian category 
$\cA (\wp)$ depending on $\wp$, and that in that context 
there is an algebraic  localization reflecting the homotopical one. 

This applies to the examples from equivariant topology. In fact, there is a homotopy
category level version which applies for $G$-spectra, but when one
takes homotopy groups to move into algebra one needs to take account
of the fact that at a subgroup $K$ we get a module over $H^*(BG/K)$
(i.e., the ambient category varies with the prime). In line with the
algebraic focus of this paper we will restrict to variation
controlled by a coefficient system $R$ of rings. 

Thus we assume our  dual system of localizations is monoidal so that
if $R$ is a ring then the dual coefficient system  $R(\wp)$ gives a
diagram of rings. This means that each localization
$L_{\wp}$ needs to have a version for $R(\fq)$-modules for all $\fq\leq
\wp$. 

\begin{defn} A {\em dual system of relative localizations} is a left
  absorbative system of functors
$$L_{\wp_1/\wp_2}:  \mbox{$R(\wp_2)$-modules}\lra
\mbox{$R(\wp_2)$-modules}$$
(where $\wp_1\geq \wp_2$), which are  transitive in the sense that
when $\wp_1\geq \wp_2\geq \wp_3$, the diagram 
$$\xymatrix{
 \mbox{$R(\wp_2)$-modules} \ar[r]^{L_{\wp_1/\wp_2}} \ar[dd]^{R_*}
&\mbox{$R(\wp_2)$-modules} \ar[d]^{R_*}\\
&\mbox{$R(\wp_3)$-modules} 
\ar[d]^{L_{\wp_2/\wp_3}} \\
\mbox{$R(\wp_3)$-modules} \ar[r]^{L_{\wp_1/\wp_3}}
&\mbox{$R(\wp_3)$-modules}
}$$
commutes. 
\end{defn}

We will see that the equivariant examples satisfy this transitivity
condition. Unfortunately this condition by itself does not seem to be
enough to complete the algebraic construction, so we will  restrict
further.

\subsection{Multiplicative systems}
Once again, we return to the equivariant setting for motivation. 

For a torus $G$ the complex  representations of $G/K$ are the
representations  $V$ of $G$ which are $K$-fixed (i.e.,  $V^K=V$). Consider 
 the  $H$-essential representations of $G/K$:
$$\Rep (G/K)_{H/K}:=\{ V\in \Rep (G/K) \st V^H=0\}.  $$
Now suppose we have a dual coefficient system $R(K)$ of rings
and Euler class functions
$$e: \Rep (G/K)\lra R(K)$$
which are multiplicative in the sense that $e(V\oplus W)=e(V)e(W)$,
and compatible with the dual coefficient system in the sense that
$$e(\infl_{G/K}^{G/L}V)=R_* e(V). $$
Now we write
$$\cE_{H/K}:=\{ e(V)\in R(K) \st V^H=0\}.  $$

\begin{remark}
(i) If $L<K$ then $V^L=0$ implies $V^K=0$ so that $\cE_L\subseteq \cE_K$,  and localization is transitive.  

(ii) If $V$ is an arbitrary representation of $G$ then 
$V=V^K\oplus V'$ with $(V')^K=0$ so that the multiplicative set of
Euler classes of $H$-essential representations of various subgroups is partially
generated by inflations:  
$$\cE_{H/1}=\langle R_*\cE_{H/K}, \cE_{K/1}\rangle . $$
\end{remark}

Abstracting this example slightly we reach the definition. 

\begin{defn}
A {\em relative system of Euler classes} for a coefficient system of
rings is specified by one
multiplicative set $\cE_{\wp_1/\wp_2}$ in $R(\wp_2)$ whenever $\wp_1 \geq \wp_2$
so that 
$$\cE_{\wp_0/\wp_2}=\langle R_* \cE_{\wp_0/\wp_1}, \cE_{\wp_1/\wp_2}\rangle . $$
\end{defn}

An any relative system of Euler classes gives a dual relative system of
localizations by taking 
$$L_{\wp_1/\wp_2}=\cEi_{\wp_1/\wp_2}. $$

This gives a sufficiently general framework that we can cover the equivariant cases by our machinery. 

\subsubsection{The dual system for all closed subgroups of a torus}
\label{subsubsec:egtorus}
 In the toral example $\fX_a$ with all subgroups,   we have
 $R(K)=H^*(BG/K)$ and for $L\leq K$ we define 
$$L_{K/L}M=\cEi_{K/L}M. $$

\subsubsection{The dual system for connected subgroups of a torus}
 In the toral example $\fX_c$ with connected subgroups,  we have
$R(K)=\cO_{\cF /K}$ and for $L\leq K$ we define 
$$L_{K/L}M=\cEi_{K/L}M. $$

\subsection{Localizations of products}
There is a second problem with permitting $M(\wp)$ to lie in an abelian
category $\cA (\wp)$ depending on $\wp$, because we need to take
products of objects from different categories. We therefore assume 
that the categories are all enriched in an abelian category $\cA$, and
that products in $\cA (\wp)$ are created in $\cA $. We may then hope
that the adelic cohomology takes values in $\cA$.

Given this, we then need to take localizations of products in the form
$$L_{\wp_1}\prod_{\wp_2\leq \wp_1}M(\wp_2).$$
Here we assume our relative system of localizations comes from a
relative system of Euler classes. 
$$L_{\wp_1/\wp_2}M(\wp_1)=\cEi_{\wp_1/\wp_2}M(\wp_1)=
\colim \left( M(\wp_1) \stackrel{e_1}\lra 
M(\wp_1)\stackrel{e_1e_2}\lra M(\wp_1) \stackrel{e_1e_2e_3}\lra\cdots\right) $$ 
We then define the localization of the product by a precisely similar colimit
$$L_{\wp_1}\prod_{\wp_2\leq \wp_1}M(\fq)=\cEi_{\wp_1}\prod_{\wp_2\leq
  \wp_1}M(\fq)=
\colim (\prod_{\wp_2\leq \wp_1}M(\wp_2) 
\stackrel{e_1}\lra 
\prod_{\fq\leq \wp}M(\wp_2) \stackrel{e_1e_2}\lra \prod_{\fq\leq
  \wp}M(\wp_2)  \stackrel{e_1e_2e_3}\lra\cdots .$$
This time there is some interpretation since if $e\in \cE_{\wp}$, for
each $\fq\leq \wp$ we may write $e=e'_{\fq}e''_{\fq}$, and in the
$\fq$ factor we interpret multiplication by $e''_{\fq}$ as an
isomorphism, so that in effect $e$ is multiplication by $e'_{\fq}$. 

\subsection{Generalized adelic cohomology}
In the context that we have
\begin{itemize}
\item a coefficient system $R$ of commutative
rings
\item  a coefficient system $M$ of $R$-modules
\item a dual system of Euler classes
\end{itemize}
then  we can define the adelic chain complex and adelic cohomology
 $H^*_{ad}(\fX; \cEi, M)$ by the same formula as before.

\section{Examples of adelic cohomology}
None of  following three examples are covered by the original
definition of adelic cohomology, but the second is covered by the
varying-category version of Section \ref{sec:varcat}. 

\subsection{Projective curves}
\newcommand{\cOC}{\cO_C}
\newcommand{\cKC}{\cK_C}
If $C$ is a smooth irreducible projective curve we may take $\fX$ to
consist of the irreducible subvarieties (i.e., the closed points of
$C$ are minimal and the generic point is maximal; by a theorem of
Thomason this is the Balmer spectrum of perfect complexes of
quasi-coherent sheaves).

We take the coefficient system to be given by the completed stalks
$(\cOC)_{\wp}^{\wedge}$ of the structure sheaf. The adelic cochain
complex is 
$$C^*_{ad}(\spcc( C))=\left(
\cKC\oplus \prod_x(\cOC)_{x}^{\wedge}
\lra \cKC  \tensor_{\cOC}\prod_x(\cOC)_{x}^{\wedge}
\right)$$
where $\cKC$ is the ring of meromorphic functions, $x$ runs through
the closed points  and $(\cOC)_{x}^{\wedge}$ is the completed stalk at
$x$. 
\begin{lemma}
For any locally free $\cOC$-module $\cF$ of finite rank, the adelic cohomology is the sheaf cohomology. 
$$H^*_{ad}(\spcc( C); \cF )=H^*(C; \cF). $$
\end{lemma}

\begin{remark}
In effect the adelic complex is the embodiment of the residue approach
to cohomology. 
\end{remark}
\begin{proof}
We may work in the category of sheaves and see that the square
$$\xymatrix{
\cOC\ar[r]\ar[d] & \cKC\ar[d]\\
\prod_x(\cOC)_{x}^{\wedge}\ar[r]& \cKC 
\tensor_{\cOC}\prod_x(\cOC)_{x}^{\wedge}
}$$
is a homotopy pullback (where now $\cKC$ is the constant sheaf of
meromorphic functions and $(\cOC)_x^{\wedge}$ is a skyscraper sheaf
concentrated at $x$). Indeed, the fibres of both horizontals are
isomorphic because the local cohomology of a ring and its completion
are isomorphic. Tensoring with a locally free sheaf $\cF$ preserves
this property. The homotopy pullback square gives a cofibre sequence
of sheaves. Taking cohomology gives the adelic complex: indeed, since
$\cKC$ is the constant sheaf $H^*(C; \cKC)=H^0(C; \cKC)
=\cKC$, and since $(\cOC)_x^{\wedge}$ is an injective skyscraper sheaf. 
$H^*((\cOC)_{x}^{\wedge})=H^0((\cOC)_{x}^{\wedge})=(\cOC)_{x}^{\wedge}$. 
\end{proof}

\subsection{Rational $SO(2)$-spectra}
As mentioned above, 
$\spcc(\mbox{$SO(2)$-spectra/$\Q$}) $
is a partially ordered set with one maximal (generic) prime
$\wp_{SO(2)}$ and   closed points $\wp_C$ corresponding to the finite
cyclic subgroups $C$. Each $\wp_C\leq \wp_{SO(2)}$ and there are no
other containments. 

The structure sheaf has value $\Q=H^*(BT/T)$ on $\wp_{SO(2)}$ and $\Q
[c]=H^*(BT/C)$ at $\wp_{C}$. The adelic complex is 
$$C^*_{ad}(SO(2); \cEi, H^*(BG/\bullet))=\left( \Q\oplus \prod_C\Q[c]\lra 
\cEi \prod_C\Q[c]\right), $$
and 
$$H^i_{ad}(SO(2); \cEi, H^*(BG/\bullet))=\dichotomy{\Q \hspace{13.5ex}\mbox{ if } i=0}
{\bigoplus_CH_*(BT/C) \mbox{ if } i=1}. $$

\subsection{Rational $O(2)$-spectra}
We describe an example where it is clear how to define an appropriate 
cohomology  but which is not covered by the version of adelic 
cohomology described here. 

The point is that the Balmer spectrum of rational $O(2)$-equivariant 
cohomology theories is not topologically discrete. In fact there is a 
homeomorphism 
$$\spcc(\mbox{$O(2)$-spectra/$\Q$}) =
\cC \coprod \mcD$$
where $$\cC=\spcc(\mbox{$SO(2)$-spectra/$\Q$}) \mbox{ and }\mcD=\{ (D_2), (D_4), (D_6), \ldots, O(2)\}$$
where $\mcD$ is topologized as the subset $\{1/n\st n\geq 1\} \cup \{0\}$ of $\R$.
 
We have already defined appropriate coefficients for 
$\spcc(\mbox{$SO(2)$-spectra/$\Q$})$, but the difference is that the
structure sheaf has stalks $H^*(BO(2)/C)=H^*(BSO(2)/C)^{C_2}=\Q [d]$ with $d=c^2$ of
codegree 4. Hence
$$H^i_{ad}(\cC ; \cEi, H^*(BO(2)/\bullet))=\dichotomy{\Q \hspace{13.5ex}\mbox{ if } i=0}
{\bigoplus_CH_*(BO(2)/C) \mbox{ if } i=1}. $$

Over $\mcD$ it makes sense to  treat the
coefficients as a  sheaf over $\mcD$ and to take sheaf cohomology
rather than simply taking the product of stalks. The relevant sheaf
for equivariant homotopy theory is the constant sheaf $\Q$, so that
$$H^0(\mcD; \Q)=C(\mcD, \Q). $$
The first sheaf cohomology is an uncountable vector space, which does
not appear relevant to $\pi^{O(2)}_*(S^0)\tensor \Q$. On the other
hand, the constant sheaf is injective in the category of realizable
sheaves, reflecting the fact that understanding realizable objects is an
important ingredient in constructing a model.

\section{Adelic cohomology and the homotopy of the sphere}

The author's original  motivation for the definition of adelic cohomology came from  the
occurrence of adelic cochains  in the study of rational torus-equivariant
cohomology theories. By tom Dieck splitting the rational stable homotopy groups of the
sphere are well known additively. If $G$ is a torus we have
$$\pi^G_*(\bbS )\cong \bigoplus_K\Sigma^{\codim (K)}H_*(BG/K). $$
In some sense this is an input to the algebraic model of
\cite{tnqcore}, so we are certainly not expecting an independent
calculation of $\pi^G_*(\bbS)$. On the other hand, the expression of $S^0$ as a homotopy
pullback of a diagram of ring spectra does show that the adelic
cohomology gives information about the ring structure. It is also
interesting to see how information about completed objects (in particular uncountable vector
spaces) feeds into the final answer (which is torsion and countable).  

\begin{prop}
For any torus $G$, using the generalized adelic coefficients of
Example \ref{subsubsec:egtorus},  there is a spectral sequence
$$H^*(\fX_c; \cEi, \cO_{\cF/\bullet} )\Rightarrow \pi^G_*(S^0)$$
\end{prop}

\begin{proof}
One of the main results of \cite{tnqcore} is the fact that the
equivariant sphere $S^0$ is
the homotopy pullback of a cubical diagram of ring $G$-spectra. 
Filtering the cube by distance from the empty face  gives a spectral
sequence
$$H^*(\pi^G_*(R(\sigma))\Rightarrow \pi^G_*(S^0),  $$
where $\sigma$ runs through dimension vectors $d_0>d_1>\cdots >d_s$ 
(i.e., it runs through non-empty subsets of the dimension poset $\{ 0<1< \ldots <r\}$). 
The definition of the adelic cochains was motivated by the
isomorphisms 
$$C^{\sigma}(\fX_c; \cEi, \cO_{\cF/\bullet} )\cong \pi^G_*(R(\sigma) ), $$
and by construction the maps in the cube correspond to the
differentials. Accordingly the spectral sequence takes the form in the statement.
\end{proof}

This raises the question of the significance of the individual
cohomology groups, and the behaviour of the spectral sequence. 

\begin{prop}
The adelic cohomology corresponds to the tom Dieck splitting
$$H^s(\fX_c; \cE, \cOcF )\cong \bigoplus_{\codim (K)=s} H_*(BG/K), $$
and the spectral sequence collapses to show
$$\pi^G_*(\bbS)=\bigoplus_s H^s(\fX_c; \cE, \cOcF )=
\bigoplus_K\Sigma^{\codim K}H_*(BG/K).$$
\end{prop}

\begin{remark}
The present proof of the collapse of this spectral sequence depends on tom Dieck
splitting. 
\end{remark}

\begin{proof}
We describe a filtration
$$ 0=F^{r+1}\subset F^r \subset \cdots \subset F^0=\mbox{Whole-$(r+1)$-Cube}$$
 of the cube so that the subquotients $\Fb^n=F^n/F^{n+1}$ have only one
 cohomology group 
$$H^n(\Fb^n)=\bigoplus_{\codim (K)=n} H_*(BG/K). $$
Indeed, we take 
$$F^n =\bigoplus_{\codim (d_s)\geq n} C^{d_0>\cdots >d_s}, $$
noting that this is a subcomplex since differentials either retain the last term subgroup or replace it by
a proper subgroup. 

We note that the quotient 
$$\Fb^n=F^n/F^{n+1} $$
is an $n$-cube, namely the one in which every term ends with a codimension $n$ subgroup. 
Thus
$$\Fb^0=\left( \cO_{\cF /G}\right), \Fb^1=\left( \prod_{\codim (H)=1}\cO_{\cF/H}\lra
 \cEi_G \prod_{\codim (H)=1}\cO_{\cF/H}\right)$$
and  
$$
\Fb^2=\left( 
\diagram 
\prod_{\codim (K)=2}\cO_{\cF/K}\rto \dto &
 \cEi_G \prod_{\codim (K)=2}\cO_{\cF/K}\dto\\
\prod_{\codim (H)=1}\cEi_H \prod_{K<H}\cO_{\cF/K}\rto &
\cEi_G \prod_{\codim (H)=1}\cEi_H \prod_{K<H}\cO_{\cF/K}
 \enddiagram \right)
$$
It remains to show that $\Fb^s$ has a single cohomology group in
codegree $s$, and to identify it. The collapse of the spectral
sequence uses the tom Dieck splitting, and a verification that
the factors in that decomposition  correspond to those at $E_2$. 
 \end{proof}

The formal ingredient is as follows. 
\begin{lemma}
Suppose $M_i$ is an $R_i$-module for each $i$, and that $\cE$ is a
multiplicative sequence consisting of elements $(r_i)\in \prod_i R_i$ 
almost all equal to 1. If each $M_i$ is $\cE$-torsion free then 
the vertical map 
$$\diagram 
\bigoplus_i M_i\rto \dto &\cEi \bigoplus_i M_i\dto\\ 
\prod_i M_i\rto &\cEi \prod_i M_i
\enddiagram$$
is a homology isomorphism. Both horizontals are injective, and the
common cokernel is isomorphic to 
$$\bigoplus_i (\cEi M_i)/M_i. $$
\end{lemma}

\begin{proof}
Considering the Snake Lemma, it suffices to prove that the map 
$$(\cEi \bigoplus_i M_i)/(\bigoplus_i M_i)\lra (\cEi \prod_i
M_i)/(\prod_i M_i) $$
is an isomorphism. This follows since elements of $\cE$ are almost all
1, and the $M_i$ are $\cE$-torsion free. 
\end{proof}

We now apply the formal ingredient in two stages.

\begin{cor}
If $M_L$ is a $\cE_{G/L}$-torsion free $\cO_{\cF/L}$-module for all $L$, 
there is a short exact sequence
$$0\lra \prod_{L<K}M_L\lra \cEi_K \prod_{L<K}M_L\lra
\bigoplus_{L<K}(\cEi_{K/L}M_L)/M_L\lra 0.$$
\end{cor}

\begin{proof}
We need only observe that $\cE_K$ satisfies the hypotheses for $\cE$
in the lemma. Then we know that the cokernel of  
$$\prod_{L<K}M_L\lra \cEi_K \prod_{L<K}M_L$$
 is the same as that of 
$$\bigoplus_{L<K}M_L\lra \cEi_K \bigoplus_{L<K}M_L=\bigoplus_{L<K}\cEi_{K/L} M_L
$$
\end{proof}

\begin{cor}
If $M_{K_s}$ is a $\cE_{G/K_s}$-torsion free $\cO_{\cF/K_s}$-module, 
there is a short exact sequence
\begin{multline*}
0\lra \prod_{K_0<K} \cEi_{K_0} \cdots
\cEi_{K_{s-1}}\prod_{K_s<K_{s-1}} M_{K_s} \lra 
\cEi_K \prod_{K_0<K} \cEi_{K_0} \cdots
\cEi_{K_{s-1}}\prod_{K_s<K_{s-1}} M_{K_s} \\
\lra 
\bigoplus_{K>K_0>\cdots >K_s}(\cEi_{K/K_s}M_{K_s})/M_{K_s}\lra 0.
\end{multline*}
\end{cor}

\begin{proof}
We apply the previous corollary $s$ times. 
\end{proof}

%\part{A homotopy pullback cube}

\section{Derived commutative algebra}
Our calculation of the adelic cohomology in commutative algebra relies 
on an argument in the derived category $\sfD (R)$. In effect, it is the 
argument of \cite{tnqcore} transposed into algebra. However there are 
some differences beyond the change of context. Firstly, the diagram of 
rings involved is slightly different (as detailed below), and secondly 
the proof has been packaged much better.

\subsection{Height, dimension and partial order}
We assume that our commutative ring $R$ is catenary and Noetherian
and of  finite dimension $r$. 
We use the Balmer ordering, so that $\fq \leq \fp$ if and only if $\fq 
\supseteq \fp$. Accordingly, $\dim (\fp)$ is the Krull dimension of $A/\fp$.
 Balmer-minimal primes are the closed points (which correspond to the ideals which are maximal under containment), and 
Balmer-maximal primes correspond to irreducible components.

%We also assume that $R$ is equidimensional in the sense that all
%maximal chains to a closed point have length $r$. We will also 
%write $\dim (\fp)=\dim (R/\fp)=r-\mathrm{height}(\fp)$. 

\subsection{Localization, completion and local cohomology}
We write $L_{\fp}M=M_{\fp}$ for the usual localization from commutative algebra.
If $\fp$ means the algebraic prime (a subset of 
the ring $R$), then the localization inverts all elements outside 
$\fp$.  We write $\fp_b=\{ M\st M_{\fp}\simeq 0\}$ for the associated 
Balmer prime and  the localization is nullification of all elements of $\fp_b$.

We use the traditional notation 
 $\Gamma_{\fp}M$ for the derived $\fp$-power torsion functor. Thus
 if $\fp =(x_1, \ldots , x_n)$ we may use the model 
$$\Gamma_{\fp}M=(R\lra R[1/x_1])\tensor_R\cdots \tensor_R (R\lra 
R[1/x_n])\tensor_RM. $$
The functor $\Gamma_{\fp}$ is the $R/\fp$-cellularization.
(We note that \cite{BIK} uses this notation for the composite 
$L_{\fp}\Gamma_{\fp}M=\Gamma_{\fp}L_{\fp}M, $
which we will never do.) We note that the composite is smashing in the
sense that  
$$\Gamma_{\fp}L_{\fp}M \simeq (\Gamma_{\fp}L_{\fp}R)\tensor_R M . $$

We use the traditional notation 
 $\Lambda_{\fp}M$ for the derived $\fp$-completion functor,  
so that if $\fp =(x_1, \ldots , x_n)$ we may use the model 
$$\Lambda_{\fp}M=\Hom_R(\Gamma_{\fp}R, M). $$
We also write
$$V_{\fp}M=\Hom_R(L_{\fp}R, M).$$
(We note that \cite{BIK} uses $\Lambda_{\fp}$  for the composite 
$V_{\fp}\Lambda_{\fp}M=\Lambda_{\fp}V_{\fp}M, $ which we will never do.) We note that
$$\Lambda_{\fp}V_{\fp}M\simeq \Hom_R(\Gamma_{\fp}L_{\fp}R, M).$$

\begin{defn}
\label{defn:suppcosupp}
The support and cosupport of an $R$-module $M$ are defined by 
$$\supp (M)=\{ \fp\st \Gamma_{\fp}L_{\fp}R\tensor_R M\not \simeq
0\}. $$
$$\cosupp(M): =\{ \fp \st V_{\fp}\Lambda_{\fp}M\not \simeq 0\}
=\{ \fp \st \Hom_R(L_{\fp}\Gamma_{\fp}R_{\fp}, M)\not \simeq 0\}. $$
\end{defn}

\begin{remark}
(i) When $M$ is compact the support is 
$$\{ \fp \st M_{\fp}\not \simeq 0\}=\{ \fp\st M\not \in \fp_b\}, $$
 but in general the support is a proper subset of $\{\fp \st
 M_{\fp}\not \simeq 0\}$. 

(ii) Since the ring is Noetherian, for any prime $\fp$ we may choose a
compact object $K_{\fp}$ so that $\supp
(K_{\fp})=\overline{\{\fp\}}$. For example if $\fp=(x_1, \ldots, x_n)$
we may take $K_{\fp}=(R\stackrel{x_1}\lra R)\tensor_R \cdots \tensor_R
(R\stackrel{x_n}\lra R)$. 
\end{remark}

The main fact we use is that an object is trivial if it has empty support or if it has 
empty cosupport. It is helpful to bear in mind that  the support and the cosupport have the same
Balmer-minimal elements  \cite[Theorem 4.13]{BIK}.

\subsection{Semiorthogonal decompositions by support}
We will want to consider collections  $\cF$ of primes closed under 
specialization (`families') and collections  $\cG$ of primes closed under 
generalization (`cofamilies'). If $\cF$ is a family, we write
$\tilde{\cF}$ for the complementary cofamily. 

In particular we consider the subgroups above and below a fixed prime  $\fq$: 
$$\Lambda (\fq )=\{ \fp\st 
\fp\leq \fq \}  \mbox{ and } V(\fq )=\{ \fp\st \fq \leq \fp \}   .$$
The first is a family (the closure of $\{ \fq\}$) and the 
second is a cofamily. 

Given a family $\cF$,  we may consider the set of Koszul objects for primes in $\cF$. Taking the cellularization with respect
to these small objects gives  $\Gamma_{\cF}X$ (so that
$\Gamma_{\fp}=\Gamma_{\Lambda (\fp)}$) and the nullification gives
$\LcFt X$ (so that $L_{\fp}=L_{V(\fp)}$). We then have a natural  cofibre sequence
$$\Gamma_{\cF}X \lra X\lra \LcFt X  $$
with 
$$\supp (\Gamma_{\cF}X)=\supp  (X) \cap \cF \mbox{  and }\supp
(\LcFt X)=\supp  (X) \setminus \cF. $$ 

If $\cF$ is the family of primes of dimension $\leq i$ and
$\tilde{\cF}$ is the complementary cofamily of primes of dimension
$\geq i+1$ we  write 
$$M_{\leq i}\lra M\lra M_{\geq i+1}$$
for the cellularization and nullification.  We say that a map $X\lra
Y$ is a $(\leq i)$-equivalence if it induces an equivalence $X_{\leq
  i}\lra Y_{\leq i}$, or equivalently if it is a equivalence when
tensored with any compact object $K$ with $\supp (K)$ consisting of
primes of dimension $\leq i$.

\section{The adelic homotopy pullback cube}
\label{sec:adeliccube}

Our analysis is based on expressing an $r$-dimensional ring as the 
homotopy pullback of an $(r+1)$-cube $C$ of simpler rings. For the integers $R=\Z$
we have the classical Hasse square 
$$\diagram 
\Z \rto \dto  & \Q \dto \\
\prod_p\Z_p^{\wedge} \rto   &  \Q \tensor 
\prod_p\Z_p^{\wedge}, 
\enddiagram$$
and in general it is a suitably arranged version of the adelic chain complex.

\subsection{The local and adelic cubes}
We are going to describe the diagram $\Rad$ of rings
which is cubical in the sense that it is a functor from an
$(r+1)$-cube $C$ to commutative rings. 
There are two natural views of the $(r+1)$-cube: as a product of copies
of $\bbI=(0<1)$ and as the set of subsets of $\{ 0, 1, \ldots ,
r\}$. The latter point of view focuses on the important features,
whilst the former helps us draw pictures. 

We consider subsets $\boldd=(d_0>d_1> \cdots >d_s)$ of $\{0,1, \ldots ,
r\}$ and view this as a flag of dimensions. We will display this
subset at a point of the cube with coordinates $(a_0, a_1, \ldots , a_r)$
where each coordinate $a_i$ takes the value $1$ if one of the
dimensions is $i$, and otherwise takes the value 0. 
%For an ideal $I$ we write
%$$R\fcomp{I} :=\prod_{\fm\leq I}R_{\fm}^{\wedge}, $$
%which is to say the product of the completions at the closed points in $V(I)$.

\begin{defn}
\label{defn:RlocRad}
The {\em adelic} diagram is defined by 
$$\Rad (\boldd):=C^{\boldd}(\spec (R); L, \Lambda R). $$
% The {\em   local} diagram is defined by 
%$$\Rloc (\boldd):=C^{\boldd}(\spec (R); L, \Lambda' R), $$
%where 
%$$\Lambda'_{\fp}R=\dichotomy{R \mbox{ if } \fp \mbox{ is not maximal}}
%{R_{\fp}^{\wedge}\mbox{ if } \fp \mbox{ is maximal}}. $$
\end{defn}

If $R$ is 1-dimensional $\Rad$ is a  minor variation on
the Hasse square. 
It is worth writing the diagram completely in the 2-dimensional
irreducible case. The layout is 
$$\diagram
&(010)\rrto \ddto&&(110)\ddto\\
(000)\rrto \ddto \urto&& (100) \ddto \urto&\\
&(011)\rrto &&(111)\\
(001)\rrto  \urto&& (101) \urto&
\enddiagram$$
%and the diagram $\Rloc$ of rings is 
%$$
%\diagram 
%&\prod_{\fp}R_{\fp}\rrto \ddto&&R_{(0)}\otimes \prod_{\fp}R_{\fp}\ddto\\
%R\rrto \ddto \urto&& R_{(0)} \ddto \urto&\\
%&\prod_{\fp}R_{\fp}\otimes \prod_{\fm \leq \fp} R_{\fm}^{\wedge}\rrto &&R_{(0)}\otimes 
%\prod_{\fp}R_{\fp}\otimes \prod_{\fm\leq \fp}R_{\fm}^{\wedge}\\
%\prod_{\fm}R \rrto  \urto&& R_{(0)} \otimes \prod_{\fm}R\urto&
%\enddiagram$$
and the diagram $\Rad$ of rings is 
$$\diagram 
&\prod_{\fp}(R_{\fp}^{\wedge})_{\fp}\rrto \ddto&&R_{(0)}\otimes \prod_{\fp}(R_{\fp}^{\wedge})_{\fp}\ddto\\
R\rrto \ddto \urto&& R_{(0)} \ddto \urto&\\
&\prod_{\fp}R_{\fp}\otimes \prod_{\fm\leq \fp}R_{\fm}^{\wedge}\rrto &&R_{(0)}\otimes 
\prod_{\fp}R_{\fp}\otimes \prod_{\fm\leq \fp}R_{\fm}^{\wedge}\\
\prod_{\fm}R_{\fm}^{\wedge} \rrto  \urto&& R_{(0)} \otimes \prod_{\fm}R_{\fm}^{\wedge} \urto&
\enddiagram$$
%Note that in the 2-dimensional case $\Rloc$ and $\Rad$ only differ in the top back horizontal. 

\subsection{Strategy}
First we recall the Cubical Reduction Principle for homotopy
pullbacks.  A cubical diagram  $X: C\lra \bbD$ is a homotopy pullback if the initial
point $X(\emptyset)$ is the homotopy inverse limit over the punctured cube
$PC$. It is thus clear that a 0-cube is a homotopy pullback if
$X(\emptyset)\simeq *$
%Now write $\bbI=(0\stackrel{e}\lra 1)$ for the
For a 1-cube $X: \bbI \lra \mcD$ write $X_f=\fibre (X(0)\lra
X(1))$ for the homotopy fibre.  This diagram is a   homotopy 
pullback if and only if the map $X(0)\stackrel{\simeq}\lra X(1)$ is an
equivalence which happens if and only if $X_f\simeq *$. 
  
Now suppose $C =\bbI \times C'$, and note that $X: C \lra \bbD$ induces
a cube $X^1_f: C'\lra \bbD$ of homotopy fibres, where the $1$ refers
to the fact that the fibre has been taken with respect to the first
coordinate.  The Cubical Reduction Principle states that 
 the diagram  $X$ is a homotopy pullback if and only if $X_f^1$ is a
 homotopy pullback.

\begin{thm}
The diagram $\Rad$ is a homotopy pullback. 
\end{thm}

\begin{proof}
For each $n$-dimensional prime we may consider the set $\Lambda (\fp)$ 
of primes in $\fp$, and form the $(n+1)$-cube indexed by subsets of 
$\{ 0, 1, \ldots, n\}$. We consider the $(n+1)$-cube $\Rad (\fp)
$, with the same definition as $\Rad$, but the primes are
restricted to $\Lambda (\fp)$ and hence the dimensions are restricted 
to $\{0, 1, \ldots, n\}$. Evidently if $\fq \leq \fp$ we have maps of diagrams
$$\Rad(\fq)\lra \Rad(\fp)\lra \Rad.$$

Note  that $\Rad$ is a homotopy pullback if and only if $\Rad 
\tensor K_{\fp}$ is a pullback for all $\fp$. Since $K_{\fp}\tensor
R_{\fq}\simeq 0$ unless $\fq\leq \fp$ we see 
$$K_{\fp}\tensor_R \Rad \simeq K_{\fp}\tensor_R\Rad (\fp), $$
so that it suffices to show $K_{\fp} \tensor_R \Rad (\fp)$ is a pullback
for all primes $\fp$.

We will prove by induction that $\dim (\fp)=n$ then $\Rad (\fp)$ is a homotopy
pullback in dimension $\leq n$. The base of the induction is the trivial case $n=-1$. 

For the inductive step we suppose that $\dim (\fp)=n$ and if $\fq\leq
\fp$ with $\dim (\fq)=i \leq n-1$  then $\Rad (\fq)$ is a homotopy
pullback in dimension $\leq i$. By the Cubical Reduction Principle, 
$\Rad (\fp)$ is a homotopy pullback  if and only if $(\Rad (\fp)
)_f^{n}$  is a homotopy pullback. 

Since $\fp$ is the only $n$-dimensional prime in $\Rad(\fp)$, the
cubical reduction takes the fibre of localization at $\fp$, and in
view of  the fibre sequence $\Gamma_{V(\fp)^c}R \lra R \lra
L_{V(\fp)}R$ we have 
$$\Rad (\fp)^{n}_f(d_0>\cdots >d_s)=(\Gamma_{V(\fp)^c}R)\tensor_R \left[\Rad (\fp) (d_0>\cdots 
>d_s)\right]. $$
Any prime $\fq\leq \fp$ of dimension $\leq n$ in $V(\fp)^c$ is actually of
dimension $\leq n-1$.  Next note that 
$$K_{\fq}\tensor \Rad (\fp) (\fq_0>\cdots >\fq_s)\simeq 0$$
unless $\fq \geq \fq_0$: this uses the fact that $K_{\fq}\tensor R_{\fq_0}\simeq 0$
unless $\fq_0\leq \fq$, and the fact that  $K_{\fq}$ is compact so
that it passes inside the products. Accordingly, 
 $$K_{\fq}\tensor \Rad (\fp)^{n}_f \simeq K_{\fq}\tensor \Rad
 (\fq),  $$
which is a pullback cube by the induction hypothesis, completing
the inductive step.

By induction we see that $K_{\fp}\tensor_R \Rad (\fp) $ is a homotopy pullback for all
primes of dimension $r$, and hence $\Rad$ is a homotopy pullback as
required. 
\end{proof}

\begin{remark}
This inductive scheme would apply equally well to other localization
systems provided $K_{\fp}\tensor A_{\fq}\simeq 0$ unless $\fq\leq
\fp$, and provided the support of the fibre of $1\lra A_{\fp}$ does not contain $\fp$.

 If $A_{\fp}=\Lambda_{\fp}L_{\fp}$ as for the
Beilinson-Parshin case the first condition is clear since $K_{\fp}$ is
small and  $K_{\fp}\tensor_R L_{\fq}R\simeq 0$ unless $\fq\leq
\fp$. For the second condition we factor it as $1\lra L_{\fp}\lra
\Lambda_{\fp}L_{\fp}$, and it suffices to show that the fibres of both
factors are supported in dimension $\leq n-1$. This is true as before
for the first map. For the second the fibre is of the form $\Hom (L_{\Lambda (\fp)^c}
R, L_{\fp}M)$, and since $K_{\fp}$ is small and $\fp \not \in \Lambda (\fp)^c\cap
V(\fp)$. 
\end{remark}

\end{document}